\documentclass[11pt,a4paper,twoside]{article}
\usepackage{bm, amsmath, amssymb, amsthm} 

\def\RR{\mathbb{R}}

\newcommand\tr{\operatorname{Trace}}
\newcommand\Div{\operatorname{div}}

\def\Ric{\operatorname{Ric}}

\def\<{\langle}
\def\>{\rangle}

\def\eq{\hspace*{-1.5mm}&=&\hspace*{-1.5mm}}

\newtheorem{corollary}{Corollary}

\newtheorem{remark}{Remark}
\newtheorem{lemma}{Lemma}
\newtheorem{proposition}{Proposition}
\newtheorem{theorem}{Theorem}

\author{Vladimir Rovenski\footnote{Department of Mathematics, University of Haifa, Mount Carmel, 31905 Haifa, Israel.
  \newline
  e-mail: {\tt vrovenski@univ.haifa.ac.il}}
  \ and \
  Robert Wolak\footnote{
  Faculty of Mathematics and Computer Science, Institute of Mathematics of Jagiellonian University,
  Lojasiewicza 6, 30-348 Krakow, Poland.
  e-mail: {\tt robert.wolak@im.uj.edu.pl}}
}

\title{On Ricci curvature of metric structures on $\mathfrak{g}$-manifolds}
\begin{document}

\date{}
\maketitle

\begin{abstract}
We study the properties of Ricci curvature of ${\mathfrak{g}}$-manifolds with particular attention paid to higher dimensional abelian Lie algebra case.
The relations between Ricci curvature of the manifold and the Ricci curvature of the transverse manifold of the characteristic foliation are investigated.
In particular, sufficient conditions are found under which the ${\mathfrak{g}}$-manifold can be a Ricci soliton or a gradient Ricci soliton.
Finally, we obtain a amazing (non-existence) higher dimensional generalization of the Boyer-Galicki theorem on Einstein K-manifolds for a special class of abelian ${\mathfrak{g}}$-manifolds.


\end{abstract}

\section*{Introduction}

In \cite{A}, D.\,Alekseevsky and P.\,Michor have began the study of
general ${\mathfrak{g}}$-manifolds, that is
smooth manifolds $M$ with an action of constant rank of a Lie algebra ${\mathfrak{g}}$, or a bit more restrictively,
a locally free action.
Several classes of ${\mathfrak{g}}$-manifolds with some additional geometrical structures have been studied in great detail.
If we take a 1-dimensional Lie algebra, a ${\mathfrak{g}}$-manifold is then just a smooth manifold with a nonvanishing vector field.
In this case we have almost contact, strict contact, (almost) contact metric structures as well as K-contact and Sasakian ones, cf. \cite{b2010,bg}.
If~we take a $p$-dimensional (with $p>1$) abelian Lie algebra, then as an example of such ${\mathfrak{g}}$-manifolds can serve a manifold equipped with $f$-{structure with parallelizable kernel, e.g., \cite{cff90,gy,yan},
that is a smooth manifold with a $p$-dimensional distribution (plane field). A special case of this structure are ${\cal K}, {\cal S}$ or ${\cal C}$ manifolds studied by many authors, e.g., cf. \cite{fip} and the references there. For $p=3$, 3-Sasakian manifolds, cf. \cite{bg}, form a special class of $so(3)$-manifolds.

 G.\,Cairns demonstrated that the study of geometrical and topological properties of totally geodesic foliations can be reduced to the study of these properties of the derived $\mathfrak{g}$-manifolds, cf. \cite{ca-toh}.
 As~the author is interested more in foliations themselves, these structures are called \textit{tangentially Lie foliations}.
Namely, a~foliation $\mathcal F$ on a Riemannian manifold $(M,g)$ is called a \textit{tangentially Lie foliation} if there is a complete Lie parallelism $\{\xi_1,\ldots,\xi_p\}$ along its leaves that preserves the horizontal subbundle ${\mathcal F}^{\perp}$.
Another term used is a tangentially $\mathfrak{g}$-foliation, where $\mathfrak{g}$ is the Lie subalgebra of ${\mathcal X}_M$ spanned by the tangential vector fields $\{X_1,\ldots,X_p\}$. For~convenience, suppose that the vector fields $X_i$ are orthonormal.
Obviously, a manifold with a tangentially Lie foliation is a ${\mathfrak{g}}$-manifold, and a ${\mathfrak{g}}$-manifold with a locally free ${\mathfrak{g}}$-action admits a tangentially ${\mathfrak{g}}$-foliation.
If~the action of the Lie algebra is not locally free then the situation is a bit more complicated.
Another context in which ${\mathfrak{g}}$-manifolds appear is the geometrical study of ODE's and PDE's on smooth manifolds.
If an ODE or PDE admits a nontrivial symmetry group, or more precisely, nontrivial infinitesimal symmetries, and these
symmetries define a regular foliation, the ambient manifold is a ${\mathfrak{g}}$-manifold. The properties of this foliated ${\mathfrak{g}}$-manifold influence the properties of solutions of our system of differential equations, and in some cases permit to solve the system effectively, cf.~\cite{olv,Wo-comp}.

In recent years, in the light of new and interesting physical applications, a lot of attention has been paid to K-contact and Sasakian structures.
They can be understood as $\mathfrak{g}$-manifolds with one-dimensional $\mathfrak{g}$ and a very special transverse structure -- K\"ahlerian.
When we try to understand higher dimensional corresponding structures, we have to take into account the algebraic structure of $\mathfrak{g}$.
In fact, such manifolds have been studied for many years within the framework of the so-called $f$-\textit{structure}  using tensor fields.
Later, a systematic study of Sasakian and 3-Sasakian manifolds as foliated manifolds with a very special characteristic foliation was carried out by Ch. Boyer and K. Galicki, for a summary of these results see \cite{bg}. The second author also made some contribution to this research, see \cite{wo00}.
In~this paper, we will use $\varphi$ instead of $f$ -- a symbol being reserved for mappings between manifolds and smooth functions depending on the context.

\section{Ricci curvature of $\mathfrak{g}$-foliated manifolds}

We are interested in those $\mathfrak{g}$-manifolds, whose characteristic foliation is regular and its dimension is equal to the dimension of the Lie algebra~$\mathfrak{g}$.
Let $G$ be a simply connected Lie group, whose Lie algebra is isomorphic to ${\mathfrak{g}}$. Then there is a natural locally free right action of $G$ on the manifold $M$, whose orbits are the leaves of the foliation $\mathcal F$, cf.~\cite{pa}.

Assume that the action is isometric. Then the Lie group $G$ of the Lie algebra $\mathfrak{g}$ can be mapped in the group of isometries ${\rm Isom}(M,g)$ of the Riemannian manifold $(M,g)$. As the manifold $M$ is compact the group ${\rm Isom}(M,g)$ is also compact. So, the closure $\bar G$ of $G$ in ${\rm Isom}(M,g)$ is compact and its orbits are the closures of leaves of our foliation.
Any bi-invariant Riemmannian metric on the group $\bar G$ restricted to $G$ defines a Riemannian metric on the tangent bundle to the leaves for which the action of $G$ is isometric. As the orthogonal complement is invariant we can extend this metric to a bundle-like metric on the manifold for which the leaves of the canonical foliation are totally geodesic.
As the Riemannian metric when restricted to the leaves is induced by a bi-invariant metric of a compact Lie group.
Corollary~1.4 in \cite{mil} ensures that the sectional curvature of our metric $g$,
\[
 K(u,v) \ge 0
\]
for any vectors $u,v \in T{\mathcal F}$. Thus, the Ricci curvature of the leaves
\[
 \Ric_{\mathcal F}(u,u) \ge 0,
\]
and the equality holds if and only if $u$ is orthogonal to $[\mathfrak{g},\mathfrak{g}]$.
Therefore, for abelian groups $K(u,v) =0$ and $\Ric_{\mathcal F}(u,u)=0$ for $u,v \in T{\mathcal F}$.

\begin{theorem}\label{T-01} Let ${\cal F}$ be a totally geodesic Riemannian foliation of a compact manifold $(M,g)$. If the leaves of ${\cal F}$ have non-negative Ricci curvature and ${\cal F}$ is modeled on a Riemannian manifold of positive Ricci curvature then the foliated manifold admits a bundle-like metric of positive Ricci curvature.
\end{theorem}

\begin{proof} It is a simple generalization of \cite[Theorem 2.7.3]{gw}.
 The tangent bundle $TM$ admits the splitting $T{\mathcal F} \oplus {\mathcal H}$, where the horizontal subbundle ${\mathcal H}$ is the orthogonal complement of $T{\mathcal F}$. As in \cite[Chapter 2]{gw}, we can define the vertical warping of the Riemannian metric $g$.
For any vector $w \in TM = T{\mathcal F} \oplus {\mathcal H}$ let $w^v$ and $w^h$ be its vertical and horizontal component, respectively.
Let $\varphi$ be any basic function on the foliated manifold $(M,{\mathcal F})$. The vertical $\varphi$-warping of the Riemannian metric is defined as follows:
\[
 g_{\varphi}(w_1,w_2) = e^{2\varphi(x)}g(w_1^v,w_2^v) + g(w_1^h,w_2^h).
\]
The formulae proved in \cite[Chapter 2]{gw} are formulated and demonstrated for Riemannian submersions with totally geodesic fibres.
They are of local character. Therefore, they are also verbatim true for Riemannian foliations with totally geodesic leaves as such foliations are locally defined by Riemannian submersions with totally geodesic fibres.
Theorem 2.7.3 of \cite{gw} is, in fact, true for Riemannian submersions of compact manifolds with totally geodesic fibres.
Therefore, its proof is also valid in our case.
\end{proof}

Using Theorem~\ref{T-01} and the considerations preceding its formulation, we obtain

\begin{corollary}
Let $M$ be a $\mathfrak{g}$-manifold given by an isometric action, and the characteristic foliation is modeled on a manifold of positive Ricci curvature.
 Then $M$ admits a Riemannian metric of positive Ricci curvature.
\end{corollary}

If the acting group $G$ on the $\mathfrak{g}$-manifold is compact, then we can construct a bundle-like metric with tangent sectional curvature non-negative,
so the orbits are submanifolds of non-negative Ricci curvature. Combining Theorem~\ref{T-01} with Myers's Theorem we get

\begin{theorem} Let a compact group $G$ act locally freely and isometrically  on a compact Riemannian manifold $(M,g)$. If the induced Riemannian foliation is modeled on a Riemannian manifold with positive Ricci curvature then the fundamental group of the  manifold $M$ is finite.
\end{theorem}

As an immediate corollary we get the following result on 3-Sasakian manifolds.

\begin{corollary}
Let $M$ be a compact 3-Sasaki-manifold. Then if the characteristic foliation is modeled on a manifold of positive Ricci curvature, then the fundamental group of $M$ is finite.
\end{corollary}

If the Lie algebra $\mathfrak{g}$ is abelian, then the same proof yields the following.

\begin{corollary}
Let $M$ be a $\mathfrak{g}$-manifold given by an isometric action with $g$ abelian, and let the characteristic foliation be modeled on a manifold of negative Ricci curvature. Then $M$ admits a Riemannian metric of negative Ricci curvature.
\end{corollary}

Corollary 2 of \cite{bn2008} permits us to conclude the following.

\begin{corollary}
The characteristic foliation of  a compact $\mathfrak{g}$-manifold given by an isometric action with abelian $\mathfrak{g}$  cannot be modeled on  a manifold of negative Ricci curvature.
\end{corollary}

\section{Ricci solitons and Einstein metrics on almost $S$-manifolds}

We are interested in ${\mathfrak g}$-manifolds with some additional structures, manifolds with ``partial" almost complex structures and associated Riemannian metrics.

An $(2n+p)$-dimensional Riemannian manifold $(M,g)$ endowed with a tensor field $\varphi$ of type $(1,1)$,
unit vector fields $\xi_i\ (1\le i\le p)$ and $1$-forms $\eta_i$, is called an \textit{almost $S$-manifold} if these tensors satisfy
\begin{eqnarray}\label{2.1}
 \varphi^2 = -I + \sum\nolimits_{\,i}\eta^i\otimes\xi_i,\quad \eta^i = g(\cdot\,,\xi_i),\\
 \label{2.1b}
 d\eta^i(X,Y)=g(X,\varphi(Y)),\quad X,Y\in\mathfrak{X}_M,
\end{eqnarray}
see \cite{gy}. From \eqref{2.1} we get the equality $\varphi^3 +\varphi = 0$.
The $1$-forms $\eta^i$ are called \textit{contact forms}, and $\xi_i$ are \textit{Reeb vector fields}
and the 2-form $F(X,Y) = g(X,\varphi(Y))$ is called the Sasaki form.
The structure $(\varphi,\xi_i,\eta^i,g)$ on $M$ is called an \textit{almost $S$-structure},
and the manifold is denoted by $M^{2n+p}(\varphi,\xi_i,\eta^i,g)$.
It follows from \eqref{2.1} and \eqref{2.1b} that
\begin{eqnarray*}
 && \varphi(\xi_i) = 0,\quad
 \eta^i \circ \varphi = 0,\quad
 {\rm rank}(\varphi)=2\,n,\\
 && g(X, \varphi(Y)) = -g(\varphi(X), Y) ,\\
 && g(\varphi(X),\varphi(Y))= g(X,Y) -\sum\nolimits_{\,i}\eta^i(X)\,\eta^i(Y)
\end{eqnarray*}
for any $X,Y\in\mathfrak{X}_M$.
The tangent bundle of $M$ is split into the sum of subbundles as $TM={\cal D}\otimes \widetilde{\cal D}$, where ${\cal D}=\varphi(TM)$
is a $2n$-dimensional distribution and $\widetilde{\cal D}={\rm span}(\xi_1,\ldots,\xi_p)$.
The restriction of $\varphi$ to ${\cal D}$ determines a complex structure on it.
From \eqref{2.1} we obtain that the metric $g$ is compatible:
\[
 g(X,Y) = g(\varphi(X), \varphi(Y)) + \sum\nolimits_{\,i}\eta^i(X)\,\eta^i(Y).
\]
 A $\varphi$-structure on a manifold $M$ is called \textit{normal} if
\[
 N_{\varphi} + 2\sum\nolimits_{\,i} d\eta^i\otimes\xi_i =0,
\]
where
 ${N}_{\varphi}(X,Y):=\varphi^2([X,Y])+[\varphi(X), \varphi(Y)] - \varphi([\varphi(X), Y]) - \varphi([X, \varphi(Y)])$
is the \textit{Nijenhuis torsion} of $\varphi$. A~normal almost ${\cal S}$-manifold is called an ${\cal S}$-\textit{manifold}.
Note that the Sasaki form of an almost $S$-manifold is closed.

\begin{remark}\rm
If \eqref{2.1} are satisfied, and instead of \eqref{2.1b}, the Sasaki form is closed the manifold $M^{2n+p}(\varphi,\xi_i,\eta^i,g)$ is called an \textit{almost} K-\textit{manifold}. A~normal almost K-manifold is called a K-\textit{manifold}.
These structures are higher dimensional versions of K-contact and Sasaki structures (or manifolds).
The main difference is that the characteristic foliation of a K-contact  manifold is totally geodesic, while in the Sasakian case it is an isometric flow.
The same is true for the higher dimensional case.  The characteristic foliation of an almost K-manifold is totally geodesic, in fact,
$\nabla_{\xi_i}\xi_j=0$, and for K-manifold the vector fields $\xi_i$ are Killing, therefore, the characteristic foliation is also Riemannian, see~\cite{b1970}.
 For the relation between K-manifolds and transversely K\"{a}hlerian foliations given by an action of an abelian Lie algebra, see \cite{dkpw00,dkw07}.
\end{remark}

When we work with a fixed base $\{\xi_1,\ldots,\xi_p\}$ of ${\mathfrak g},$ then let
\[
 h_i = \frac12\,{\cal L}_{\xi_i}\,\varphi.
\]

\noindent
 In addition to the D.E.~Blair paper, \cite{b1970}, in several other papers important tensorial equalities for almost ${\cal S}$-{manifolds} have been demonstrated, cf. \cite{cff90,dip01}.
 In particular, see \cite[Propositions~3.3 and 3.4]{BP2008},
 the operators $h_i$ are self-adjoint, anti-commute with $\varphi$ and for any $i,j=1,\ldots,p$,
\begin{subequations}
\begin{eqnarray}\label{E-four-1}
 && h_i\,\xi_j =0,\\
\label{E-four-2}
 && \nabla_X\,\xi_i = -\varphi(X) - \varphi(h_i X), \quad X\in TM,\\
\label{E-four-3}
 && \nabla_{\xi_i}\,\varphi =0,\\
\label{E-four-5}
 && h_i\circ\varphi=-\varphi\circ h_i,\\
\label{E-sol-02}
 && \tr h_i = 0,\quad \tr (\varphi\circ h_i) = 0.
\end{eqnarray}
\end{subequations}
From (\ref{E-four-1},b) we have $\nabla_{\xi_i}\,\xi_j = 0$.
The second equality in \eqref{E-sol-02} follows from \eqref{E-four-5}.

A complete Riemannian manifold $(M,g)$ with a vector field $X$ is called a \textit{Ricci soliton} if
\begin{equation}\label{E-sol-01}
 \frac12\,{\cal L}_X\,g + \Ric +\,\lambda\,g = 0
\end{equation}
for some $\lambda\in\RR$ and $X\in\mathfrak{X}_M$.
 If the constant ${\lambda} >0$ it is called shrinking, steady if ${\lambda} =0$, and expanding if ${\lambda} <0$.
 If~the vector field $X$ is the gradient of some function $f, X = \nabla f,$ then it is called a gradient
 Ricci soliton and it satisfies the equation
\[
 \Ric + \,{\rm Hess}\,f = \lambda\,g.
\]
In \cite{pw} the authors demonstrated, cf. Corollary 2.2, where $\nabla$ is the Levi-Civita connection of $g$.

\begin{theorem} If $Z$ is a Killing vector field on a gradient soliton $(M,g,f)$, then either $\nabla_Z f=0$
or we have an isometric splitting $M=N\times\RR$ where $N$ is a gradient soliton with the same radial curvature as $M$.
\end{theorem}

Applying the above theorem a finite number of times we get the following

\begin{theorem}
Let $(M,g,f)$ be endowed with a ${\mathfrak{g}}$-foliation. Then $(M,g,f)$ is isometric to a gradient Riemannian soliton
$(N \times \RR^s,\, g_1\times g_0, \bar{f}) $ where $(N,g_1,\bar{f})$ is a gradient Ricci soliton and $N$ is a $\bar{\mathfrak{g}}$-manifold,
for which the function $\bar{f}$ is basic for a foliation of dimension $p-s$.
\end{theorem}

\begin{corollary}
If a compact $\mathfrak{g}$-manifold is a gradient Ricci soliton for a function $f$, then $f$ is a basic function.
\end{corollary}

For an ${\cal S}$-manifold, $\xi_i$ are Killing vector fields and the distribution ${\cal D}$ is totally geodesic.
The following lemma generalizes this fact.
Let $H$ be the mean curvature vector field of ${\cal D}$, then
\[
 \Div\xi_\alpha=-\<H,\xi_\alpha\>,\quad
 H=\sum\nolimits_{\,\alpha} \<H,\xi_\alpha\>\,\xi_\alpha.
\]

\begin{lemma}\label{L-H=0} For an almost ${\cal S}$-manifold we have
$\Div\xi_\alpha=0$ and $H=0$, i.e., the distribution ${\cal D}$ is harmonic.
\end{lemma}

\begin{proof}
Using (\ref{E-four-2},d) for a local orthonormal frame $(E_i)$ of ${\cal D}$, we have
\begin{equation*}
 \<H,\xi_\alpha\>=\sum\nolimits_{\,i}\, g(\nabla_{E_i}\,\xi_\alpha,\,E_i) = -\tr\,\varphi -\tr(\varphi\circ h_\alpha) = 0 ,
\end{equation*}
thus the claim follows.
\end{proof}

Define $\bar\eta=\sum_{\,i}\,\eta^i$ and $\bar\xi=\sum_{\,i}\,\xi_i$.

\begin{lemma} For an almost ${\cal S}$-manifold we have
\begin{subequations}
\begin{eqnarray}\label{Equ-1-2}
 \Div \varphi \eq -2\,n\,\bar\eta,\\
\label{Equ-1-3}
 \Ric(\xi_i, X) \eq -\nabla^*\nabla\,\xi_i + 2\,n\,\bar\eta(X).
\end{eqnarray}
\end{subequations}
\end{lemma}

\begin{proof}
For \eqref{Equ-1-2} we apply divergence to \cite[Eqn.~(5)]{BP2008}:
\[
 (\nabla_X\,\varphi)Y=g(\varphi(X),\varphi(Y))\,\bar\xi+\bar\eta(Y)\,\varphi^2(X).
\]
For \eqref{Equ-1-3} see \cite[Eqn.~(2.6)] {tan}, where we use Lemma~\ref{L-H=0}.
\end{proof}

\begin{proposition}
On an almost ${\cal S}$-manifold, we have
\begin{eqnarray}\label{E-sol-03b}
 (\Div(h_i\circ\varphi)) X = \Ric(\xi_i, X) -2\,n\,\bar\eta(X).
\end{eqnarray}
\end{proposition}

\begin{proof} Using \eqref{E-four-2} and (\ref{Equ-1-2},b), we obtain \eqref{E-sol-03b},
see also \cite[p.~838]{bs1990} in coordinate notation.
\end{proof}

We generalize result of \cite{cs-2010}.

\begin{theorem}\label{T-05}
Let $(M,g)$ be a compact almost ${\cal S}$-manifold such that $g$ is a Ricci soliton with a nonzero potential vector field $X$ tangent to
$\ker\varphi$. Then $(M,g)$ is an Einstein manifold.
\end{theorem}

\begin{proof}
Substituting $X=u^i\xi_i$ (the summation sign is omitted) in \eqref{E-sol-01}, and using \eqref{E-sol-02}, we obtain
\begin{equation}\label{E-sol-04}
 Y(u^i)\xi_i +\<\xi_i,Y\>\nabla u^i +2\,u^i h_i\,\varphi(Y) + 2\Ric(Y) + 2\,\lambda\,Y =0 ,
\end{equation}
where $Y$ is an arbitrary vector field on $M$.
Let $(E_i)$ for $i\le 2\,n+p$ be a local orthonormal frame on $M$ that extends the set $(\xi_i)$ for $i\le p$, i.e., $E_i=\xi_i$ for $i\le p$.
Contracting \eqref{E-sol-04} yields
\begin{equation}\label{E-sol-05}
 \xi_i(u^i) = -s -(2\,n+p)\lambda ,
\end{equation}
where $s$ is the scalar curvature of $(M,g)$. Differentiating \eqref{E-sol-04} along a vector field $Z$ and contracting (i.e., $Y=Z=E_i$) provides
\begin{equation}\label{E-sol-06}
 \xi_i(Y u^i) +\<\nabla_{\nabla u^i}\,\xi_i, Y\> -\eta^i(Y)\Delta u^i +2(h_i\,\varphi(Y)) u^i + 2\,u^i[\Ric(Y,\xi_i) -2\,n\,\eta^i(Y)] =0 ,
\end{equation}
where we used \eqref{E-sol-03b} and notation $\Delta u^i = -\tr({\rm Hess}(u^i))$.
Substituting $\xi_j$ for $Y$ in \eqref{E-sol-06} yields
\begin{equation}\label{E-sol-07}
 \xi_i(\xi_j\, u^i) -\Delta u^j -2\,u^i \tr(h_i\circ h_j) + \xi_j(s) = 0.
\end{equation}
The use of \eqref{E-sol-05} in \eqref{E-sol-07} and commutativity $\xi_i(\xi_j u^i)=\xi_j(\xi_i\,u^i)$ provides
\begin{equation}\label{E-sol-08}
 \Delta\, u^j = -2\,u^i \tr(h_i\circ h_j) .
\end{equation}
From \eqref{E-sol-08}, after summation we obtain
\begin{equation}\label{E-sol-09}
 \Delta\big(\sum\nolimits_j(u^j)^2\big) = -2\sum\nolimits_j\,\|\nabla h_j\|^2 - 4\,\|h_V\|^2 ,
\end{equation}
where $\|h_V\|^2=u^j u^i \tr(h_i h_j)\ge0$.
Integrating \eqref{E-sol-09} on $M$ and using the Divergence Theorem yield $u^i={\rm const}$ for all $i$.
Thus, $X$ is a Killing field and \eqref{E-sol-01} reduces to $\Ric = -\lambda\,g$, i.e., $(M,g)$ becomes an Einstein manifold.
\end{proof}

Ch. Boyer and K. Galicki \cite{bg2001} proved that a K-contact manifold, which is an Einstein manifold, is a Sasaki-Einstein manifold.
In the higher dimensional case the situation slightly different.

Let $M^{2n+p}(\varphi,\xi_i,\eta^i,g)$ be an almost  ${\cal S}$-manifold with $p>1$.
Put $\bar{\xi} = \sum_{\,i}\xi_i$ and $\bar{\xi}_i =\xi_i -\xi_1$ for $i=2,\ldots,p$.
The distribution $\bar{\cal D} ={\cal D}\oplus \<\bar{\xi}\>$ is integrable  and the vector fields
$\bar{\xi}_2,\ldots,\bar{\xi}_p$ are orthogonal to $\bar{\cal D}$, so is  the foliation $\bar{\mathcal F}$ defined by these Killing vector fields.
Both foliations are Riemannian and totally geodesic.
By the deRham decomposition theorem, cf. \cite{KN}, applied to $\bar{\cal D}$ and $\bar{\cal F}$),
locally, $(M,g)$ is a Riemannian product of leaves of these foliations with the induced metrics.
Moreover, the Killing vector fields $\bar{\xi}_2,\ldots,\bar{\xi}_p$ are parallel, thus if $M$ is compact, then $\Ric(\bar{\xi}_i,\bar{\xi}_i) = 0$ for any $i=2,\ldots,p,$  cf. \cite[Theorem 3]{bn2008}.
As an easy consequence of these consideration we get the following.

\begin{theorem}\label{T-06}
There are no Einstein compact almost $S$-manifolds.
\end{theorem}

\begin{proof}
Let $M^{2n+p}(\varphi,\xi_i,\eta^i,g)$ be an almost $S$-manifold.
The Riemannian metric $g$ satisfies the equation $\Ric(X,Y) = \lambda\, g(X,Y)$ for any vectors $X,Y$ and some real constant $\lambda$.
Choosing any $X=Y=\bar{\xi}_i$, we get $\lambda =0$. Therefore, the Riemannian manifold $(M,g)$ is Ricci flat.
Next, by the Bochner Theorem, cf. \cite[Theorem 1.1]{fw75}, the Killing field $\bar{\xi}$ is a parallel vector field.
This property together with the fact that the distribution ${\cal D}$ is $\bar{\xi}$-invariant ensure that $d\bar{\eta}=0$ -- a~contradiction.
\end{proof}

Combining Theorems~\ref{T-05} and \ref{T-06}, we obtain the following:

\begin{corollary}
There is no compact almost $S$-manifold $M^{2n+p}(\varphi,\xi_i,\eta^i,g)$ that is a Ricci soliton for a vector field $X$ tangent to $\ker\varphi$.
\end{corollary}


\end{document}